\documentclass[a4paper]{article}

\usepackage[mathscr]{eucal}
\usepackage{amssymb}
\usepackage{latexsym}
\usepackage{amsthm}
\usepackage{amsmath}
\usepackage[dvips]{graphicx}
\usepackage{psfrag}

\newtheorem{theorem}{Theorem}[section]

\newtheorem{lemma}[theorem]{Lemma}
\newtheorem{corollary}[theorem]{Corollary}
\theoremstyle{definition}

\newtheorem{problem}[theorem]{Problem}

\makeatletter
\@addtoreset{equation}{section}

\makeatother

\title{Geometrical approach to Seidel's switching for strongly regular graphs}
\author{Hiroshi Nozaki}

\begin{document}
\maketitle
\begin{center}
	Graduate School of Information Sciences,
	Tohoku University \\
	Aramaki-Aza-Aoba 09,
	Aoba-ku,
	Sendai 980-8579,
	Japan\\
	nozaki@ims.is.tohoku.ac.jp
\end{center}
\renewcommand{\thefootnote}{\fnsymbol{footnote}}
\footnote[0]{Supported by JSPS Research Fellowship.}
\footnote[0]{The author stay at the University of Texas at Brownsville from August 24th, 2009 to August 23rd, 2010}

	\begin{abstract}
	In this paper, we simplify the known switching theorem due to Bose and Shrikhande as follows.   
	Let $G=(V,E)$ be a primitive strongly regular graph with parameters $(v,k,\lambda,\mu)$. 
	Let $S(G,H)$ be the graph from $G$ by switching with respect to a nonempty $H\subset V$.
	Suppose $v=2(k-\theta_1)$ where $\theta_1$ is the nontrivial positive eigenvalue of the $(0,1)$ adjacency matrix of $G$. This strongly regular graph is associated with a regular two-graph. 
	Then, $S(G,H)$ is a strongly regular graph with the same parameters if and only if the subgraph induced by $H$ is $k-\frac{v-h}{2}$ regular. Moreover, $S(G,H)$ is a strongly regular graph with the other parameters if and only if the subgraph induced by $H$ is $k-\mu$ regular and the size of $H$ is $v/2$. We prove these theorems with the view point of the geometrical theory of the finite set on the Euclidean unit sphere. 
	\end{abstract}

\section{Introduction}
A simple graph $G=(V,E)$ is called a strongly regular graph with parameters $(v,k,\lambda,\mu)$ if the cardinality of $V$ is $v$, $G$ is $k$ regular, any two adjacent vertices are adjacent to $\lambda$ common vertices, and any two nonadjacent vertices are adjacent to $\mu$ common vertices. The complement of a strongly regular graph is also strongly regular. A strongly regular graph is said to be primitive if both it and its complement are connected. It is known that an imprimitive strongly regular graph is either a complete multipartite graph, or the disjoint union of a number of copies of a complete graph. Primitive strongly regular graphs are known as association schemes of class $2$, or distance regular graphs of diameter $2$. The $(0,1)$ adjacency matrix of a graph $G$ is defined by the matrix indexed by the vertices, whose $(x,y)$ entry is $1$ if $x$ is adjacent to $y$, and $0$ otherwise. Let $A_1$ be a $(0,1)$ adjacency matrix of a strongly regular graph, and $A_2$ be that of the complement. Then, the identity matrix $I$, $A_1$ and $A_2$ generate the commutative algebra, called the Bose-Mesner algebra. Let $E_i$ ($i=0,1,2$) be the primitive idempotents of the Bose-Mesner algebra, where $E_0:= J/v$, $E_i E_j= \delta_{i,j} E_i$, $J$ is the all one matrix, and $\delta_{i,j}$ is the Kronecker's delta. We can write $A_i$ and $E_i$ as linear combinations of each other, namely $A_i=\sum_{i=0}^2 p_i(j)E_j$ and $E_i=\frac{1}{v}\sum_{i=0}^j q_i(j) A_j$. $P=(p_i(j))$, whose $(j+1,i+1)$ entry is $p_i(j)$, and $Q=(q_i(j))$, whose $(j+1,i+1)$ entry is $q_i(j)$, are called the first and second eigenmatrices, respectively. $P$ and $Q$ give a lot of information of the strongly regular graph, but both of them depend only on parameters $(v,k,\lambda,\mu)$. However, even if two strongly regular graphs have the same parameters, there is a possibility that they are not isomorphic to each other. For examples, we have the exactly four strongly regular graphs with parameters $(28,12,6,4)$, those are the triangular graph $T(8)$ and three other graphs called Chang graphs \cite{Chang}. 

By Seidel's switching of edges of a strongly regular graph associated with a regular two-graph, we may get new strongly regular graphs with the same parameters. Bose and Shrikhande determine the conditions of the switching set which may give a new strongly regular graph \cite{Bose-Shrikhande}. In this paper, we simplify the conditions of the switching set. 
Let $S(G,H)$ denote the graph from $G$ by switching with respect to $H \subset V$. 
Let $k> \theta_1 > \theta_2$ denote the eigenvalues of $(0,1)$ adjacency matrix of a strongly regular graph. 
The following are the main theorems in this paper. 
	\begin{theorem} \label{main}
	Let $G=(V,E)$ be a primitive strongly regular graph with parameters $(v,k,\lambda,\mu)$. 
	$H$ is a subset of $V$, and its cardinality is $h$.
	Suppose $v=2(k-\theta_1)$. 
	Then, the following are equivalent:
	
	{\rm (i)} $S(G,H)$ is a strongly regular graph with parameters $(v,k,\lambda,\mu)$.	
	
	{\rm (ii)} The subgraph induced by $H$ is $k-\frac{v-h}{2}$ regular. 
	\end{theorem}  
	
	\begin{theorem} \label{main2}
	Let $G=(V,E)$ be a primitive strongly regular graph with the parameters $(v,k,\lambda,\mu)$. 
	$H$ is a subset of $V$.  
	Suppose $v=2(k-\theta_1)$. Then, the following are equivalent.
	
		{\rm (i)} $S(G,H)$ is a strongly regular with parameters $(v,k+c,\lambda +c,\mu +c)$ where $c=v/2-2 \mu$. 
		
		{\rm (ii)} The cardinality of $H$ is equal to $v/2$ and the subgraph induced by $H$ is $k-\mu$ regular.
	\end{theorem} 
For examples, the complements of $T(8)$ or the Chang graphs are able to apply Theorem \ref{main}. In particular, by Theorem \ref{main}, we can find at least $100000$ strongly regular graphs with parameters $(276,140,58,84)$. It is known that there are at least $7715$ strongly regular graphs with these parameters \cite{Haemers-Tonchev}.  We prove these theorems with the view point of the geometrical theory of the finite subset on the Euclidean unit sphere.

\section{Preliminaries} 
In this section, we introduce some basic terminology and results. More details may be found, for examples, in \cite{Bannai-Ito}, \cite{Brouwer}, \cite{Delsarte-Goethals-Seidel}, \cite{Lint-Seidel}, \cite{Seidel} and \cite{Spence}. 

\subsection{Regular two-graphs}
Let $V$ be a set of vertices, and $\Delta$ be a collection of $3$-subsets of $V$, where an $n$-subset means a subset whose cardinality is $n$. $(V,\Delta)$ is called a two-graph if each $4$-subset of $V$ contains an even number of elements of $\Delta$. A two-graph $(V,\Delta)$ is said to be regular if each $2$-subset of $V$ is contained in a constant number of $\Delta$. 

Given a simple graph $G=(V,E)$, the set $\Delta$ of $3$-subsets of $V$ whose induced subgraph in $G$ contains an odd number of edges give rise to a two-graph $(V,\Delta)$. In fact, every two-graph can be represented in this way. Switching from $G$ with respect to a subset $H \subset V$ consists of interchanging adjacency and non-adjacency between $H$ and its complement $V \setminus H$. Graphs $G_1=(V,E_1)$ and $G_2=(V,E_2)$ represent the same two-graph if they are related by the equivalence relation of switching with respect to some $H \subset V$. An equivalence class of graphs under switching is called a switching class. Thus, a two-graph can be identified with the switching class of a graph. 

The $(0,-1,1)$ adjacency matrix $B$ of a graph $G$ is defined by the matrix indexed by the vertices, whose $(x,y)$ entry is $-1$ if $x$ is adjacent to $y$, $1$ if $x$ is not adjacent to $y$, and $0$ if $x=y$. The $(0,-1,1)$ adjacency matrix of $S(G,H)$ is $D_H B D_H$, where $B$ is the $(0,-1,1)$ adjacency matrix of $G$ and $D_H$ is the diagonal matrix whose $(x,x)$ entry is $-1$ for $x \in H$ and $1$ for $x \in V \setminus H$. Because $B$ and $D_H B D_H$ have the same eigenvalues, the eigenvalues of a two-graph are the eigenvalues of the $(0,-1,1)$ adjacency matrix of any graph in its switching class. A two-graph $(V,\Delta)$ is regular if and only if it has two distinct eigenvalues $\rho_1>0>\rho_2$, where $\rho_1 \rho_2=1-|V|$. 

The switching class of a graph $G$, and hence any two-graph, can be represented geometrically as a set of equiangular lines. Let $-\rho<0$ be the smallest eigenvalue of the $(0,-1,1)$ adjacency matrix $B$ of $G$ on $v$ vertices, and suppose that $-\rho$ has multiplicity $v-d$. Then $\rho I+ B$ is positive semidefinite of rank $d$ and so can be represented as the Gram matrix of the inner products of $n$ vectors in Euclidean space $\mathbb{R}^d$, which implies the equiangular line having the same angle $\phi$ with $\cos \phi=1/\rho$. Here, the Gram matrix of a finite set $X \subset \mathbb{R}^d$ is indexed by $X$, and its $(x,y)$ entry is the usual inner product of $x$ and $y$. Conversely, given a set of $v$ nonorthogonal equiangular lines in $\mathbb{R}^d$, there exists a two-graph from which it can be constructed  by this method. The cardinality $v$ of such a set satisfies $v\leq d(\rho^2-1)/(\rho^2-d)$, and this bound is achieved if and only if the corresponding two-graph is regular. 

The matrix $ I+ \frac{1}{\rho} B$ is the Gram matrix of a finite set $X$ on $S^{d-1}$. Similarly, we can get the finite set $X_H \subset S^{d-1}$ with the Gram matrix $I+ \frac{1}{\rho} D_H B D_H$ from $S(G,H)$. 
We define the bijection $\varphi:V \rightarrow X$. The switching with respect to $H$ means that we move $\varphi(H)$ to the antipodal part $- \varphi(H)$ in the spherical embedding. Namely, we have $X_H=(X \setminus \varphi(H)) \cup (-\varphi(H))=\{x \in X \mid x \not\in \varphi(H) \} \cup \{-x \mid x \in \varphi(H)\}$.

\subsection{Embedding to the unit sphere}
Let $G=(V,E)$ be a primitive strongly regular graph. A primitive strongly regular graph is identified with a symmetric association scheme $(V,\{R_0,R_1,R_2\})$ with two classes, where $R_0=\{(x,x) \mid x \in V \}$ and $R_1=\{(x,y) \mid (x,y) \in E \}$. Let $A_i$ be the $(0,1)$ adjacency matrix with respect to the relation $R_i$, $E_i$ be the primitive idempotents, and $m_i$ be the rank of $E_i$. 
As is well known, the spherical embedding of $V$ with respect to $E_i$ $(i=1,2)$ in the unit sphere $S^{m_i-1}$ are defined as follows. We identify $x \in V$ with the vectors $\bar x=\sqrt{\frac{|V|}{m_i}} E_i e_x$, where $e_x= ^t(0,\ldots ,0,1,0, \ldots, 0) \in \mathbb{R}^v$ with the $x$-th coordinate $1$. If the strongly regular graph is primitive, then this embedding is faithful. The standard inner product $\langle \bar x, \bar y \rangle$ in $\mathbb{R}^{m_i}$ is given by $q_i(j)/m_i=p_j(i)/k_j$ if $(x,y) \in R_j$, where $p_j(i)$ and $q_i(j)$ are entries of the first and second eigenmatrix,  $m_i=q_i(0)$ and $k_j=p_j(0)$. Namely, this spherical embedding has the structure of the strongly regular graph. We know the properties of this embedding as $s$-distance sets and spherical $t$-designs.

We introduce the concept of $s$-distance sets and spherical $t$-designs. 
Let $X$ be a nonempty finite subset of $S^{d-1}$. Define $A(X):=\{\langle x,y \rangle \mid x,y \in X, x\ne y\}$, that is, the set of the standard inner products of distinct vectors of $X$. $X$ is called an $s$-distance set if $|A(X)|=s$. $X$ is called a spherical $t$-design on $S^{d-1}$, if $\sum_{x \in X} f(x)=0$ for any $f \in {\rm Harm}_l (\mathbb{R}^d)$ with $1\leq l \leq t$, where ${\rm Harm}_l (\mathbb{R}^d)$ is the linear space of harmonic homogeneous polynomials of degree $l$, with $d$ variables. 

If $X$ is an $s$-distance set and a spherical $t$-design, and $t \geq 2s-2$, then $(X,\{R_i\})$ is an association scheme of class $s$, where $R_i=\{(x,y) \in X \times X \mid \langle x,y \rangle = \alpha_i \}$, and $A(X)=\{\alpha_1,\alpha_2,\ldots , \alpha_s \}$ \cite{Delsarte-Goethals-Seidel}. 
In particular, $X$ is a $2$-distance set and a spherical $2$-design, then $X$ has the structure of a strongly regular graph. 

For a fixed $x \in X \subset S^{d-1}$, we define $n_i(x):=|\{y \in X \mid \langle x,y \rangle =\alpha_i \}|$.
If $t \geq s-1$, then $X$ is distance invariant, that is, $n_i(x)$ is a constant number $k_i$ for any $x\in X$ \cite{Delsarte-Goethals-Seidel}. 

Let $\varphi_i$ be the embedding bijection from the vertex set of a primitive strongly regular graph $G=(V,E)$ to $S^{m_i-1}$ with respect to $E_i$ $(i=1,2)$. $\frac{|V|}{m_i}E_i$ is the Gram matrix of the spherical embedding because $^tE_i E_i=E_i$. Since we can write $E_i:=\frac{1}{|V|} \sum_{j=0}^2 q_i(j)A_j$, $\varphi(V)$ is a $2$-distance set. Moreover, it is known that $\varphi_i(V)$ is a spherical $2$-design on $S^{m_i-1}$ \cite{Cameron-Goethals-Seidel}.

\section{Known switching theorems}
Bose and Shrikhande proved the following theorem in $1970$.

\begin{theorem}[Theorem 8.1 in \cite{Bose-Shrikhande}] \label{BS-1}
	Let $G=(V,E)$ be a strongly regular graph with the parameters $(v,k,\lambda,\mu)$ where $2 k -v/2 =\lambda + \mu$. 	
	Let $H_1$ be a subset of $V$, and $H_2:=V \setminus H_1$. Let $v_i$ be the cardinalities of $H_i$. 
	Then, the following are equivalent. 
	
		{\rm (i)} $S(G,H_1)$ is strongly regular. 
		
		{\rm (ii)} The subgraph induced by $H_1$ is $w_1$ regular and the subgraph induced by $H_2$ is $w_2$ regular where 
		\[
		w_1-w_2=\frac{v_1-v_2}{2}.
		\]
\end{theorem}
Note that if there exists nonempty $H$ such that $S(G,H)$ is strongly regular, then $G$ has the condition $2 k -v/2 =\lambda + \mu$ \cite{Bose-Shrikhande}. In particular, when $S(G,H)$ is strongly regular with the same parameters, we have the following theorem. 

\begin{theorem}[Theorem 8.3 in \cite{Bose-Shrikhande}] \label{BS-2}
	Let $G=(V,E)$ be a strongly regular graph with the parameters $(v,k,\lambda,\mu)$ where $2 k -v/2 =\lambda + \mu$. 	
	Let $H_1$ be a subset of $V$, and $H_2:=V \setminus H_1$. 
	Then, the following are equivalent. 
	
		{\rm (i)} $S(G,H_1)$ is strongly regular. 
		
		{\rm (ii)} In $G$ each vertex in $H_1$ is adjacent to exactly half of vertices in $H_2$, and each vertex in $H_2$ is adjacent to exactly half of vertices in $H_2$. 
\end{theorem}
It is known that $G$ is a strongly regular graph with $2k-v/2=\lambda+\mu$ if and only if $G$ is a strongly regular graph with $v=2(k-\theta_1)$ or $v=2(k-\theta_2)$ \cite{Brouwer-Cohen-Neumaier}. If $G$ has the condition $v=2(k-\theta_2)$, then the complements $\bar{G}$ has the condition $v=2(k-\bar{\theta_1})$ where $\bar{\theta_1}$ is the positive eigenvalue of $\bar{G}$. Therefore, without loss of generality, we may assume $v=2(k-\theta_1)$. Hence, Theorem \ref{main} is the simplification of Theorem \ref{BS-2}. The switching class of a regular two-graph may contain strongly regular graphs with at most two parameters sets \cite{Brouwer-Cohen-Neumaier, Brouwer-Haemers}. Theorem \ref{main2} is the simplification of Theorem \ref{BS-1} in the case where $S(G,H)$ has the other parameters. 

\section{Proof of Theorem \ref{main}}
Let $G=(V,E)$ be a primitive strongly regular graph with parameters $(v,k, \lambda, \mu)$. Let 
$k$, $\theta_1$ and $\theta_2 $ be the eigenvalues of the $(0,1)$ adjacency matrix of $G$, where $k> \theta_1>0> \theta_2$. $G$ is identified with an association scheme $(V,\{R_0,R_1,R_2\})$, where $A_1$ is the $(0,1)$ adjacency matrix of $G$. Then, we can write the first eigenmatrix 
	\[
	P=\left[ \begin{array}{ccc}
	1& k        & v-1-k\\
	1& \theta_1 & -1-\theta_1\\
	1& \theta_2 & -1-\theta_2\\
	\end{array} \right]
	\]
where 
	\[
	\{\text{$\theta_1 $},\text{$\theta_2 $}\}=\left\{\frac{\lambda -\mu +\sqrt{(\lambda -\mu )^2-4(\mu -k)}}{2},\frac{\lambda -\mu -\sqrt{(\lambda -\mu )^2-4(\mu -k)}}{2}\right\}.
	\]
Assume $v=2(k-\theta_1)$. Then, we can determine $\theta_1=k-v/2$ and $\theta_2=\lambda-\mu-k+v/2=k-2 \mu$. Since $\theta_1>0$, we have $k > \frac{v}{2}$. Let $B$ be the $(0,-1,1)$ adjacency matrix of $G$, and $-\rho<0$ be the minimum eigenvalues of $B$. Since $B=-A_1+A_2$, the eigenvalues of $B$ are $-p_1(j)+p_2(j)$ with $j=0,1,2$. Because we have $-p_1(0)+p_2(0)=k-(v-1-k)=-1-2\theta_1$, the eigenvalues of $B$ are $-1-2\theta_1$ and $-1-2\theta_2$. Therefore, this strongly regular graph is obtained in the switching class of a regular two-graph. Then, $-\rho=-1-2\theta_1$.

Let $\varphi_2$ be the spherical embedding bijection from $V$ to $S^{m_2-1}$ with respect to $E_2$. The following is a key lemma to prove the Theorem \ref{main}. 

	\begin{lemma} \label{lemma}
	Let $G=(V,E)$ be a strongly regular graph with parameters $(v,k,\lambda,\mu)$, where  $v=2(k-\theta_1)$ (resp.\ $v=2(k-\theta_2)$). 
	Let $B$ be the $(0,-1,1)$ adjacency matrices of $G$, and $-\rho_1<0$ (resp.\ $\rho_2>0$) be the minimum (resp.\ maximum) eigenvalue of $B$. 
	Let $X$ be a finite set with the Gram matrix $I+\frac{1}{\rho_1}B$ (resp.\ $I-\frac{1}{\rho_2}B$).
	Then, $X$ coincides with the spherical embedding with respect to $E_2$ (resp.\ $E_1$).
	\end{lemma}
	
	\begin{proof}
	Suppose $G=(V,E)$ has the condition $v=2(k-\theta_1)$.
	Let $A_i$ and $E_i$ be defined above. 
	Note that for $i=0,1$, 
		\[
		E_i (\rho_1 I + B)=E_i(\rho_1 I -A_1+A_2) =\rho_1 E_i - p_1(i) E_i+p_2(i) E_i=0. 
		\]
	Therefore, $I + \frac{1}{\rho_1}B$ is equal to $\frac{m_2}{v}E_2$. We can prove the case $v=2(k-\theta_2)$ in the same manner.
	\end{proof}
By Lemma \ref{lemma}, the following is clear. 
	\begin{corollary} \label{coro}
	Let $X$ be the finite set defined in Lemma \ref{lemma}.
	Then, $X$ is a $2$-distance set and a spherical $2$-design on $S^{m_2-1}$ (resp.\ $S^{m_1-1}$), where $m_2$ (resp.\ $m_1$) is the multiplicity of the negative (resp.\ non trivial positive) eigenvalue of $A_1$.
	\end{corollary}
Now we prove Theorem \ref{main}.
	\begin{proof}[Proof of Theorem \ref{main}]
	Let $G=(V,E)$ be a primitive strongly regular graph with parameters $(v,k,\lambda,\mu)$, where  $v=2(k-\theta_1)$. Let $B$ be the $(0,-1,1)$ adjacency matrix of $G$, $-\rho$ be the minimum eigenvalue of $B$, and $v-m_2$ is the multiplicity of $-\rho$. Let $X$ be a finite set on $S^{m_2-1}$ with the Gram matrix $I+\frac{1}{\rho}B$. Then, $X$ is a $2$-distance set and a spherical $2$-design on $S^{m_2-1}$ by Corollary \ref{coro}. Let $\varphi$ be the spherical embedding bijection $V \rightarrow X$ with respect to $E_2$.
	
	First, suppose $S(G,H)=(V_H,E_H)$ is a strongly regular graph with the same parameters $(v,k,\lambda,\mu)$. Then, the $(0,-1,1)$ adjacency matrix of $S(G,H)$ is $D_H B D_H$, where $D_H$ is defined above. Let $X_H$ be the finite set on $S^{m_2-1}$ with the Gram matrix $I+\frac{1}{\rho} D_H B D_H$. Similarly, $X_H$ is a $2$-distance set and a spherical $2$-design on $S^{m_2-1}$. Note that $X_H=(X \setminus \varphi(H)) \cup (-\varphi(H))$. 
	
	Since $X$ is a spherical $2$-design, for any $f_1\in {\rm Harm}_1(\mathbb{R}^{m_2})$, 
		\begin{equation} \label{eq1}
		0=\sum_{x \in X} f_1(x)=\sum_{x \in X \setminus \varphi(H)}f_1(x)+\sum_{x \in \varphi(H)}f_1(x).
		\end{equation}
	On the other hand, for any $f_1\in {\rm Harm}_1(\mathbb{R}^{m_2})$, 
		\begin{equation} \label{eq2}
		0=\sum_{x \in X_H} f_1(x)=\sum_{x \in X \setminus \varphi(H)}f_1(x)+\sum_{x \in -\varphi(H)}f_1(x)=\sum_{x \in X \setminus \varphi(H)}f_1(x)-\sum_{x \in \varphi(H)}f_1(x)
		\end{equation}
	because $X_H$ is a spherical $2$-design and $f_1$ is a homogeneous polynomial of degree $1$. 
	By equations (\ref{eq1}) and (\ref{eq2}), we have $\sum_{x \in \varphi(H)}f_1(x)=0$ for any $f_1\in {\rm Harm}_1(\mathbb{R}^{m_2})$. Therefore, $\varphi(H)$ is a spherical $1$-design. Since $\varphi(H)$ is a $1$- or $2$-distance set and a spherical $1$-design, $\varphi(H)$ is distance invariant. Thus, the subgraph induced by $H$ is $n$ regular for some $n$. It is known that $\varphi(H)$ is a spherical $1$-design if and only if $\sum_{x \in \varphi(H)} \sum_{y \in \varphi(H)} \langle x,y \rangle=0$. Therefore, $1-\frac{1}{\rho}n+\frac{1}{\rho}(h-n-1)=0$, and hence $n=\frac{h-1+\rho}{2}=k-\frac{v-h}{2}$. 

	Second, suppose the subgraph induced by $H \subset V$ is $k-\frac{v-h}{2}$ regular. Clearly, $\sum_{x \in \varphi(H)} \sum_{y \in \varphi(H)} \langle x,y \rangle=0$. Therefore, $\varphi(H)$ is a spherical $1$-design on $S^{m_2-1}$. 
	
	Let $X_H$ be the finite set on $S^{m_2-1}$ with the Gram matrix $I+\frac{1}{\rho} D_H B D_H$. 
	Then, $X_H:=(X \setminus \varphi(H)) \cup (-\varphi(H))$, and $X_H$ has the structure of $S(G,H)$. 
	
	For any $f_1 \in {\rm Harm}_1(\mathbb{R}^{m_2})$, 
		\begin{equation}
		\sum_{x \in X_H}f_1(x)=\sum_{x \in X \setminus \varphi(H)}f_1(x) + \sum_{x \in -\varphi(H)} f_1(x)=\sum_{x \in X \setminus \varphi(H)}f_1(x) + \sum_{x \in \varphi(H)} f_1(x)=\sum_{x \in X}f_1(x)=0,
		\end{equation}
	because $\varphi(H)$ and $X$ are spherical $1$-designs.
	For any $f_2 \in {\rm Harm}_2(\mathbb{R}^{m_2})$, 
		\begin{equation}
		\sum_{x \in X_H}f_2(x)=\sum_{x \in X \setminus \varphi(H)}f_2(x) + \sum_{x \in -\varphi(H)} f_2(x)=\sum_{x \in X \setminus \varphi(H)}f_2(x) + \sum_{x \in \varphi(H)} f_2(x)=\sum_{x \in X}f_2(x)=0
		\end{equation}
	because $X$ is a spherical $2$-design on $S^{m_2-1}$, and $f_2$ is a homogeneous polynomial of degree $2$. Therefore, $X_H$ is a spherical $2$-design on $S^{m_2-1}$. Since $X_H$ is a $2$-distance set and a spherical $2$-design on $S^{m_2-1}$, $S(G,H)$ is a strongly regular graph. Since $X_H$ is a spherical $1$-design and $A(X_H)=A(X)$, $S(G,H)$ is $k$ regular. This implies that $S(G,H)$ has the same parameters $(v,k,\lambda,\mu)$.   
	\end{proof} 
We give some remarks of Theorem \ref{main}. 

Since $k-\frac{v-h}{2}$ is an integer, $h \equiv v \mod 2$. 

Suppose $H_1 \subset V$ and $H_2 \subset V$ hold the conditions in Theorem \ref{main}, (ii). 
Let $S_1$ and $S_2$ be the subgraph induced by $H_1$ and $H_2$, respectively. 
If there exists $g \in {\rm Aut}(G)$, such that ${S_1}^g=S_2$, then $S(G,H_1)$ is isomorphic to $S(G,H_2)$. 

$S(G,H_1)$ may be isomorphic to $S(G,H_2)$, even when there does not exist $g \in {\rm Aut}(G)$, such that ${S_1}^g=S_2$. 

We introduce another proof of Theorem \ref{main}. The author got this proof from a personal communication by A.E.\ Brouwer \cite{Brouwer-p,Muzychuk-Klin}. 

Let $G=(V,E)$ is a strongly regular graph defined in Theorem \ref{main}. 

First, we suppose $S(G,H)$ is a strongly regular graph with the same parameters. Let $A$ be the $(0,1)$ adjacency matrix of $G$, which is partitioned according to $\{H,V \setminus H\}$. Namely, 
	\[
	A=\left[ \begin{array}{cc}
	A_H & C \\
	^tC   & A_{V \setminus H} 
	\end{array} \right]
	=
	\left[ \begin{array}{cc}
	A_{1,1} & A_{1,2} \\
	A_{2,1}   & A_{2,2} 
	\end{array} \right]
	\]
where $A_H$ ($=A_{1,1}$) is the $(0,1)$ adjacency matrix of the subgraph induced by $H$, and $A_{V\setminus H}$ ($=A_{2,2}$) is that by $V \setminus H$. Switching with respect to $H$ implies replacing $C$ to $J-C$ in $A$.  Since $S(G,H)$ is also $k$ regular, the number of entries $1$ in $C$ is equal to that in $J-C$. Thus, the number of entries $1$ in $C$ is equal to $h(v-h)/2$. Let $f_{i,j}$ denote the average row sum of $A_{i,j}$. Then, $F=(f_{i,j})$ is called the quotient matrix. Since  the  number of entries $1$ in $C$ is $h(v-h)/2$, we can get 
	\[
	F=\left[ \begin{array}{cc}
	k-\frac{v-h}{2} & \frac{v-h}{2} \\
	\frac{h}{2}   & k-\frac{h}{2} 
	\end{array} \right].
	\]
The eigenvalues of $F$ are $k$ and $k-v/2$, because the row sums are $k$ and the trace is $2k -v/2$. It is known that the eigenvalues of $A$ interlace the eigenvalues of $F$ \cite{Brouwer-Cohen-Neumaier}. Namely, $k \geq \theta_1 \geq k-\frac{v}{2} \geq \theta_2$.  Since $v=2(k -\theta_1)$, the interlacing is tight ({\it i.e.}\ $\theta_1=k-v/2$). Therefore, this partition is equitable ({\it i.e.}\ the row sum of each $A_{i,j}$ is constant), namely, the subgraph induced by $H$ is $k-(v-h)/2$ regular \cite{Brouwer-Cohen-Neumaier}.  

Second, suppose the subgraph induced by $H$ is $k-(v-h)/2$ regular. Then, the quotient matrix $F$ is the same above. Hence, the interlacing of eigenvalues of $A$ and $F$ is tight, and hence the partition is equitable. Therefore, $S(G,H)$ is regular, and hence $S(G,H)$ is a strongly regular graph \cite{Goethals-Seidel}. Moreover $S(G,H)$ has the same parameters as that of $G$.   
  
\section{Proof of Theorem \ref{main2}}
The following is a key result in order to prove Theorem \ref{main2}.

\begin{theorem} \label{key thm}
	Let $G=(V,E)$ be a strongly regular graph with $v=2(k-\theta_1)$. 
	If there are $H \subset V$ such that $S(G,H)$ is a strongly regular graph with the other parameters. 
	Then, the spherical embedding with respect to $E_2$ is on two parallel hyperplanes of dimension at most $m_2-1$. 
\end{theorem}

\begin{proof}
	Let $H$ be a subset of $V$. Suppose that $S(G,H)$ is a strongly regular graph with the other parameters. 
	Note that if $G$ has the eigenvalues $k$, $\theta_1$ and $\theta_2$, then $S(G,H)$ has the eigenvalues $k^{\ast}$, $\theta_1$ and $\theta_2$,  where $k^{\ast}=k+v/2-2 \mu$ is degree of $S(G,H)$. Let $m_i$ be the multiplicities of $\theta_i$ as eigenvalues of $G$, and $m_i^{\ast}$ be those of $S(G,H)$. Then, $m_1+1=m_2^{\ast}$ and $m_2=m_2^{\ast}+1$.
	Since $S(G,H)$ is in the switching class of the same regular two-graph as that of $G$, $S(G,H)$ has the condition $v=2(k^{\ast}-\theta_2)$. By Lemma \ref{lemma}, the primitive idempotent of $S(G,H)$ is 
	\[
		\frac{v}{m_1^{\ast}}E_1^{\ast}= I+\frac{1}{\rho^{\ast}}A_1^{\ast}-\frac{1}{\rho^{\ast}}A_2^{\ast}
	\]
	where $\rho^{\ast}=-1-2 \theta_2$, and $A_1^{\ast}$ and $A_2^{\ast}$ are the $(0,1)$-adjacency matrix of $S(G,H)$ and that of the complement, respectively. 
	$D_H E_2 D_H$ is in the Bose-Mesner algebra of $S(G,H)$. Then, 
	\begin{align*}
		D_H E_2 D_H + E_1^{\ast}&=\frac{m_2}{v} \left( I-\frac{1}{\rho} A_1^{\ast} + \frac{1}{\rho}A_2^{\ast} \right) + \frac{m_1^{\ast}}{v} \left( I+\frac{1}{\rho^{\ast}} A_1^{\ast} - \frac{1}{\rho^{\ast}}A_2^{\ast} \right) \\
					&= \frac{m_2+m_1^{\ast}}{v}I+\frac{ m_1^{\ast} \rho -m_2 \rho^{\ast}}{v\rho \rho^{\ast}} ( A_1^{\ast}-A_2^{\ast})\\
					&= \frac{m_2+m_1+1}{v}I+\frac{ (m_1+1)(1+2 \theta_1) +m_2 (1+2\theta_2)}{v\rho \rho^{\ast}} ( A_1^{\ast}-A_2^{\ast})\\
					&=I
	\end{align*}
	where $\rho=1 + 2\theta_1$. Therefore, $D_H E_2 D_H$ is equal to $E_0^{\ast}+E_2^{\ast}$. Since
	\[
	E_2 D_H j=D_H(E_0^{\ast}+E_2^{\ast})j=v D_H j
	\]
	where $j$ is the all one column vector, the spherical embedding with respect to $E_2$ is on two hyperplanes which are perpendicular to $D_H j$.

\end{proof}
The finite set with the Gram matrix $D_H E_2 D_H$ is on one hyperplane of dimension $m_2-1$, and is identified with the spherical embedding with respect to $E_2^{\ast}$.  
Since the spherical embedding $X$ with respect to $E_2$ is a spherical $1$-design, $\sum_{x \in X} x = 0$ and hence the cardinality of the switching set $H$ is equal to $v/2$ by Theorem \ref{key thm}.

\begin{proof}[Proof of Theorem \ref{main2}]
First, suppose $S(G,H)$ is strongly regular with the other parameters. 
Then, the cardinality of $H$ is equal to $v/2$. Let $A$ be the $(0,1)$ adjacency matrix of $G$, which is partitioned according to $\{H,V \setminus H\}$. Namely, 
	\[
	A=\left[ \begin{array}{cc}
	A_H & C \\
	^t C   & A_{V \setminus H} 
	\end{array} \right]
	=
	\left[ \begin{array}{cc}
	A_{1,1} & A_{1,2} \\
	A_{2,1}   & A_{2,2} 
	\end{array} \right]
	\]
where $A_H$ ($=A_{1,1}$) is the $(0,1)$ adjacency matrix of the subgraph induced by $H$, and $A_{V\setminus H}$ ($=A_{2,2}$) is that by $V \setminus H$. By Theorem \ref{BS-1}, the both subgraphs induced by $H$ and $V \setminus H$ are $n$ regular for some integer $n$. The number of entries $1$ in $C$ is $v(k-n)/2$. 
After switching with respect to $H$, $C$ becomes $J-C$. Therefore, the number of entries $1$ in the $(0,1)$ adjacency matrix of $S(G,H)$ is $v(2n-k+v/2)$. On the other hand,  $S(G,H)$ is $k+v/2-2\mu$ regular, and the number of entries $1$ in  the $(0,1)$ adjacency matrix of $S(G,H)$ is $v(k+v/2-2\mu)$. Thus, $n=k-\mu$.

Second, suppose $H$ is $k-\mu$ regular and its cardinality is $v/2$.
Let $A$ be the $(0,1)$ adjacency matrix of $G$, which is partitioned according to $\{H,V \setminus H\}$.  The quotient matrix is 
	\[
	F=\left[ \begin{array}{cc}
	k-\mu & \mu \\
	\mu   & k-\mu 
	\end{array} \right].
	\]
Then, the eigenvalues of $F$ are $k$ and $k-2 \mu =\theta_2$. This interlacing is tight, and hence this partition is equitable. Hence, $S(G,H)$ is a strongly regular graph whose degree $k+v/2-2\mu$.
\end{proof}

We introduce another method of determining the cardinality of the switching set $H$ in Theorem \ref{main2} (ii) \cite{Brouwer-p, Muzychuk-Klin}.

By Theorem \ref{BS-1} (ii), the subgraph induced by $H_1$ is $w_1$ regular, and hence each vertex $H_1$ is adjacent to $k-w_1$ vertices
of $H_2$. After switching, each vertex in $H_1$ is adjacent to $w_1$ vertices in $H_1$, and to $v-v_1-(k-w_1)$ vertices in $H_2$. Therefore, each vertex in $H_1$ is adjacent to $v-v_1-k+2w_1$ vertices in $S(G,H_1)$. Hence, $k+v/2-2 \mu =v-v_1-k+2w_1$ and 
	\begin{equation} \label{equa1}
	w_1=k-\mu-v/4+v_1/2. 
	\end{equation}
Similarly, each vertex in $H_2$ is adjacent to $v-v_2-k+2w_2$ vertices in $S(G,H_1)$ after switching. 
Since $v_2=v-v_1$ and $k+v/2-2 \mu =v-v_2-k+2w_2$, we have 
	\begin{equation} \label{equa2}
	w_2=k-\mu+v/4-v_1/2. 
	\end{equation}
By counting the number of edges between $H_1$ and $H_2$, we have $v_1(k-w_1)=v_2(k-w_2)$. Therefore, by equations (\ref{equa1}) and (\ref{equa2}), we get $v_1(\mu+v/4-v_1/2) = (v-v_1)(\mu-v/4+v_1/2)$, {\it i.e.,} $(v_1-v/2)(v/2-2 \mu) =0$. The case $v/2-2 \mu=0$ corresponds to $c=0$. Thus, $v_1=v/2$.

\section{Applications}
When $v \leq 280$, the known strongly regular graphs with $v=2(k-\theta_1)$ have the following parameters \cite{Brouwer}. 
	\begin{multline*}
	\{(v,k,\lambda,\mu)\}= \{(10, 6, 3, 4), (16, 10, 6, 6), (16, 9, 4, 6), (26, 15, 8, 9), (28, 
  15, 6, 10), (36, 21, 12, 12),\\
	(36, 20, 10, 12), (50, 28, 15, 
  16), (64, 36, 20, 20), (64, 35, 18, 20), (82, 45, 24, 25), (100, 55,
   30, 30),\\
	(100, 54, 28, 30), (120, 68, 40, 36), (120, 63, 30, 
  36), (122, 66, 35, 36), (126, 75, 48, 39), \\
	(126, 65, 28, 39), (136, 75, 42, 40), (136, 72, 36, 40), (144, 78, 42, 42), (144, 77, 40, 
  42),\\ 
	(170, 91, 48, 49), (176, 105, 68, 54), (176, 90, 38, 54), (196, 104, 54, 56), (210, 110, 55, 60), \\
	(226, 120, 63, 64), (256, 136, 72, 72), (256, 135, 70, 72), 
	(276, 140, 58, 84), (280, 144, 68, 80)\}
	\end{multline*}
The strongly regular graphs with $v \leq 36$ in the above list have been classified. If a regular two-graph has been classified, then we may classify the corresponding strongly regular graphs. When $v\geq 50$ the classifications of regular two-graphs are not known, except a regular two-graph on $276$ vertices. Indeed, a regular two-graph on $276$ vertices is unique \cite{Goethals-Seidel}. Moreover, Goethals and Seidel \cite{Goethals-Seidel} gave one strongly regular graph with parameters $(276,140,58,84)$ in the switching class of the regular two-graph on $276$ vertices.  By Theorem \ref{main}, we can easily construct new strongly regular graphs with $(276,140,58,84)$. Indeed, a $6$-clique holds the conditions in Theorem \ref{main}. 
By the algebra software Magma, we can easily get the set of all $6$-cliques. 
And, we make the set of $6$-cliques up to transitiveness. It is easy to make the strongly regular graphs with the same parameters by switching with respect to the $6$-clique. 
New strongly regular graphs are also applicable to this method. By repeating this method, we can efficiently get new examples. We found at least $100000$ pairwise non-isomorphic strongly regular graphs with parameters $(276,140,58,84)$. However, we have not succeeded the classification of strongly regular graphs with parameters $(276,140,58,84)$, because there are too many induced subgraphs holding the conditions in Theorem \ref{main}. Since the disjoint union of spherical $1$-designs is also a spherical $1$-design, the disjoint union of $6$ clique subgraphs also satisfy the condition in Theorem \ref{main}. We can guess $100000$ strongly regular graphs are very small part of the classification.     

\begin{problem}
The strongly regular graphs with parameters $(276,140,58,84)$ are pseudogeometric $(5,27,3)$-graph. Is there a geometric strongly regular graph with there parameters? (please see \cite{Soicher} for the terminologies)
\end{problem}

\textbf{Acknowledgements.}
The author would like to thank Professor Andries E.\ Brouwer, Professor Vladimir D.\ Tonchev and Professor Akihiro Munemasa for providing useful comments and informations.


\begin{thebibliography}{9}
	\bibitem{Bannai-Ito}
	E.\ Bannai and T.\ Ito,
	{\it Algebraic Combinatorics I}, 
	Benjamin/Cummings, 1984.
	
	\bibitem{Bose-Shrikhande}
	R.C.\ Bose and  S.S.\ Shrikhande, 
	Graphs in which each pair of vertices is adjacent to the same number $d$ of other vertices,	
	{\it Studia Sci.\ Math.\ Hungar.} 5 (1970), 181--195.
	
	\bibitem{Brouwer-p}
	A.E.\ Brouwer,
	personal communication.
	
	\bibitem{Brouwer}
	A.E.\ Brouwer,
	Strongly regular graphs, 
	in: {\it The CRC Handbook of Combinatorial Designs, Second Edition} (eds.: Colbourn and Dinitz), 852--868, CRC Press, 2006.
	
	\bibitem{Brouwer-Cohen-Neumaier}
	A.E.\ Brouwer, A.M.\ Cohen and A.\ Neumaier, 
	{\it Distance-Regular Graphs}, 
	Springer-Verlag, 1989.  
	
	\bibitem{Brouwer-Haemers}
	A.E.\ Brouwer and W.H.\ Haemers, Spectra of graphs, course notes, http://www.win.tue.nl/$\sim$aeb/. 
	
	\bibitem{Brouwer-Lint}
	A.E.\ Brouwer and J.H.\ Lint, 
	Strongly regular graphs and partial geometries,
	{\it Enumeration and design (Waterloo, Ont., 1982)}, 85--122, {\it Academic Press, Toronto, ON}, 1984. 
	
	\bibitem{Cameron-Goethals-Seidel}
	P.J.\ Cameron, J.M.\ Goethals and J.J.\ Seidel, 
	Strongly regular graphs having strongly regular subconstituents,
	{\it J.\ Algebra} 55 (1978), 257--280.  
	
	\bibitem{Chang}
	L.C.\ Chang,
	Association schemes of partially balanced block designs with parameters $v=28$, $n_1=12$, $n_2=15$ and $p_{1,1}^2=4$,
	{\it Sci.\ Record} 4 (1960), 12--18.
	
	\bibitem{Delsarte-Goethals-Seidel}
	P.\ Delsarte, J.M.\ Goethals, and J.J.\ Seidel, 
	Spherical codes and designs, {\it Geom.\ Dedicata} 6 (1977), no.\ 3, 363--388. 
	
	\bibitem{Goethals-Seidel}
	J.M.\ Goethals and J.J.\ Seidel,
	The regular two-graph on $276$ vertices,
	{\it Discrete Math.} 12 (1975) 143--158.
	
	\bibitem{Haemers-Tonchev}
	W.H.\ Haemers and V.D.\ Tonchev, 
	Spreads in strongly regular graphs, 
	{\it Des.\ Codes Cryptogr.} 8 (1996), no.\ 1-2, 145--157. 

	\bibitem{Jorgensen-Klin}
	L.K.\ Jorgensen and M.\ Klin, 
	Switching of edges in strongly regular graphs.\ I. A family of partial difference sets on 100 vertices, 
	{\it Electron. J. Combin.} 10 (2003), Research Paper 17, 31 pp.\ (electronic). 
	
	\bibitem{Lint-Seidel}
	J.H.\ Lint and J.J.\ Seidel,
	Equilateral point sets in elliptic geometry,
	{\it Nederl.\ Akad.\ Wetensch.\ Proc.\ Ser.\ A 69=Indag.\ Math.} 28 (1966), 335--348. 
	
	\bibitem{Muzychuk-Klin}
	M.\ Muzychuk and M.\ Klin,
	On graphs with three eigenvalues,  
	{\it Discrete Math.} 189 (1998), no.\ 1-3, 191--207. 

	
	\bibitem{Seidel}
	J.J.\ Seidel,
	A survey of two-graphs,
	{\it Colloquio Internazionale sulle Teorie Combinatorie (Rome, 1973), Tomo I}, 481--511. Atti dei Convegni Lincei, No.\ 17, {\it Accad.\ Naz.\ Lincei, Rome}, 1976.
	
	\bibitem{Soicher}
	L.H.\ Soicher, 
	Is there a McLaughlin geometry? 
	{\it J.\ Algebra} 300 (2006), no.\ 1, 248--255. 

	
	\bibitem{Spence}
	E.\ Spence,
	Two-graphs,
	in: {\it The CRC Handbook of Combinatorial Designs, Second Edition} (eds.: Colbourn and Dinitz), 875--882, CRC Press, 2006.
	
	\end{thebibliography}
\end{document}